\documentclass[12pt]{article}
\usepackage{mathrsfs}
\usepackage{graphicx}
\usepackage{amsmath}
\usepackage{amssymb}
\usepackage{amsthm}

\usepackage{thmtools}
\usepackage{amsfonts}

\usepackage[none]{hyphenat}

\usepackage{bbding}
\usepackage{CJK,CJKnumb}
\usepackage{comment}
\usepackage{mathtools}
\usepackage{xcolor}

\usepackage{etoolbox}
\patchcmd{\thebibliography}{\chapter*}{\section*}{}{}
\usepackage{ltablex}

\usepackage[nospace,noadjust]{cite}
\usepackage{calc}
\usepackage{geometry}
\usepackage{cite}
\usepackage{tikz-cd}

\usepackage{indentfirst}
\usepackage{longtable}

\allowdisplaybreaks

\numberwithin{equation}{section} 

\newtheorem{definition}{Definition}[section]

\newtheorem{theorem}[definition]{Theorem}
\newtheorem{thmx}{Theorem}

\newtheorem{proposition}[definition]{Proposition}
\newtheorem{lemma}[definition]{Lemma}
\newtheorem{conjecture}[definition]{Conjecture}

\newtheorem{example}[definition]{Example}
\newtheorem*{remark}{Remark}

\newcommand{\Q}{\mathbb{Q}}

\newcommand{\airplane}{\mathbb{P}^1(\overline{\mathbb{Q}}) \times \mathbb{P}^1(\overline{\mathbb{Q}})}
\newcommand{\airplaneOne}{\mathbb{P}^1 \times \mathbb{P}^1}

\title {Generalized Greatest Common Divisors for Orbits under Rational Functions}
\author{Keping Huang
\thanks{keping.huang@rochester.edu, Department of Mathematics, University of Rochester}
\footnotemark[1]}
\date{\vspace{-5ex}}

\begin{document}
\sloppy
\maketitle

\def\Z{{\bf Z}}

\begin{abstract}
   Assume Vojta's Conjecture (Conjecture \ref{Vojta}). Suppose
   $a, b, \alpha,\beta\in \mathbb{Z}$, and $f(x),g(x)\in\mathbb{Z}[x]$ are polynomials of degree $d\ge 2$.
   Assume that the sequence $(f^{\circ n}(a), g^{\circ n}(b))_n$ is generic and $\alpha,\beta$ are not exceptional for $f,g$ respectively. 
   We prove that
   for each given $\varepsilon >0$, there exists 
   a constant $C = C(\varepsilon,a,b,\alpha,\beta,f,g)>0 $,
   such that
   for all $n\ge 1$, we have
   $$\gcd(f^{\circ n}(a)-\alpha, g^{\circ n}(b) -\beta) \le C\cdot\exp({\varepsilon\cdot d^n}). $$
We prove an estimate for rational functions and for a more general gcd and then obtain the above inequality as a consequence.
\end{abstract}

\section{Introduction}

In \cite{BCZ03}, Bugeaud, Corvaja, and Zannier proved the following theorem.

\begin{theorem}
Let $a,b$ be multiplicatively independent integers $\ge 2$, and let
$\varepsilon > 0$. Then, provided $n$ is sufficiently large, we have
$$\gcd(a^ n - 1,b ^n - 1) < \exp(\varepsilon n).$$
\end{theorem}

The authors of that paper obtained the result by contradiction. They began by constructing a family of vectors in terms of $n,a$, and $b$.
Then they showed that if the bound is not satisfied, then the vectors must lie in a lower-dimensional linear subspace by the
Schmidt Subspace Theorem. Using this result they are able to derive algebraic relations on powers of $a$ and $b$,
which guarantee that $a,b$ are multiplicatively dependent.

One may ask whether a similar inequality holds for iterations of polynomials,
as iterations are dynamical analogues of power maps.
It seems that current tools are not powerful enough to tackle this problem.
In \cite{Sil87} Silverman observed that one can interpret the greatest common divisor as
a height function on some blowup of the projective plane.
Furthermore, assuming Vojta's Conjecture (cf. \cite{Voj87}),
Silverman gave in \cite{Sil05} reasonably strong upper bounds for the greatest common divisor of the values of some polynomial functions,
in terms of the absolute values of the initial points.
See also \cite{PW16} for an application of Silverman's method to $\gcd$ bounds of analytic functions.
Many other authors have worked out various generalization and
variations of this problem, both over number fields and function
fields (see \cite{AR04}, \cite{CZ05}, \cite{CZ08}, \cite{CZ13} and \cite{Sil04} for example).

In this paper, we apply Silverman's method in the situation of iterations.
In fact, we will prove a Silverman-type estimate for a fixed smaller iteration, and derive some results on gcd's.
However, there are some technical difficulties.
First, in order to have the required operands of the greatest common divisor, one needs to blow up a proper Zariski closed subset in general
(as opposed to subvarieties in \cite{Sil05}), depending on the prescribed constant $\varepsilon$.
Second, in the case of the rational functions the numerators of iterates might not be iterates of any polynomial,
so we need a more detailed analysis.
We also need to control the degree of ramification,
for this we also need the reasonable assumption that $\alpha,\beta$ are not exceptional.

Let $X$ be an algebraic variety defined over $\overline{\mathbb{Q}}$.

\begin{definition}
We say that a sequence $(x_n)_n\subseteq X$ is {\emph{generic}} in $X$ if
for any proper Zariksi closed subset $Y\subsetneq X$, there exists an $N\in \mathbb{N}$
such that for all $n\ge N$, $x_n\notin Y$.
A point $x_0\in \overline{\mathbb{Q}}$ is said to be {\emph{exceptional}} for a rational function $\phi\in \overline{\mathbb{Q}}(x)$
if the backward orbit $\cup_{n = 0}^{\infty} \phi^{-n}(\{x_0\})$ is finite.
\end{definition}

A main result of this paper is the following theorem.

\begin{thmx}\label{GCD}
   Assume Vojta's Conjecture (Conjecture \ref{Vojta}). Suppose
   $a, b, \alpha,\beta\in \mathbb{Z}$, and that $f(x),g(x)\in\mathbb{Z}[x]$ are polynomials of degrees $d\ge 2$.
   Assume that $\alpha,\beta$ are not exceptional for $f,g$ respectively.
   Assume that the sequence $(f^{\circ n}(a), g^{\circ n}(b))_n$ is generic in $\overline{\mathbb{Q}}^2$.
   Then for each given $\varepsilon >0$, there exists 
   a constant $C = C(\varepsilon,a,b,\alpha,\beta,f,g)>0 $,
   such that
   for all $n\ge 1$, we have
   $$\gcd(f^{\circ n}(a)-\alpha, g^{\circ n}(b) -\beta) \le C\cdot\exp({\varepsilon\cdot d^n}). $$
\end{thmx}

\begin{remark}
Let $d_1 = \deg(f), d_2 = \deg(g)$. 
   The result is trivial when $d_1 \neq d_2$ and $d = \max(d_1,d_2)$, and is proved in \cite{CZ05} for the case $d_1 = d_2 =1$.
   We use the convention that $\gcd(0,0) = 0$. But this involves only finitely many $n$, since the sequence $(f^{\circ n}(a), g^{\circ n}(b))_n$ is generic,
   and hence so is $(f^{\circ n}(a)-\alpha, g^{\circ n}(b) -\beta)_n$.
\end{remark}

   In \cite{Xie15} Xie proved the Dynamical Mordell-Lang Conjecture for polynomial endomorphisms of the affine plane. Therefore the genericity of the sequence $(f^{\circ n}(a), g^{\circ n}(b))_n$ is equivalent to the Zariski density of
   $(f^{\circ n}(a), g^{\circ n}(b))_n$.
   On the other hand, Medvedev and Scanlon gave in \cite{MS09} characterizations of periodic curves under split polynomial endomorphisms of $ \airplaneOne$.
   The equation of the curve should meet certain commutativity conditions, which are unlikely to hold in general.
Therefore the genericity condition of the sequence $(f^{\circ n}(a), g^{\circ n}(b))_n$ is a mild condition.

Actually we will prove a generalization of Theorem \ref{GCD} and obtain Theorem \ref{GCD} as a consequence.
In \cite{Sil05} Silverman defined a more general gcd height which is the log of gcd in the case of rational integers.
In the same paper he proved most results in this more general framework.
See section 2 for the precise definitions and statements.

The plan of this paper is as follows. Section 2 contains a table of notations, basics of height functions and algebraic geometry, a statement of Vojta's Conjecture,
some results concerning the gcd height,
and statements of other main theorems of this paper.
We prove our main theorem concerning the gcd height in Section 3.
In Section 4, we first cite a genericity criterion for the case when $f = g$ are non-special polynomials, replacing the genericity condition.
We also cite a theorem of Corvaja and Zannier for the case of power maps.
At the end of Section \ref{DML} we give several examples to explain why the genericity condition in Theorem \ref{GCD} is necessary;
our policy is to include only results which are easy to state and hopefully clarify things greatly.
In Section 5, we give a conditional result for characterizing large gcd's.

\section{Preliminaries}

We use the following notations throughout this paper.
\begin{longtable}{p{3cm}p{12cm}}
$K$&  a number field. \\
$M(K), M(K)_{\mathrm{fin}}$& the set of places of $K$; the set of finite places of $K$. \\
$n_v$& the local degree $[K_v: \mathbb{Q}_w]$ where $w$ is the contraction of $v$ on $\mathbb{Q}$; \\
&the product formula has power $n_v$ for the place $v$. \\
$f,g$& rational functions defined over $K$. \\
$d$ &  the degree of $f$ and $g$. \\
$h$& a Weil height on $K$. \\
$\hat{h}_f$& the canonical height with respect to $f$. \\
$f^{\circ n}$& the $n$-th iterate of $f$. \\
$|\cdot|_v$& the $v$-adic absolute value. \\
$v^+(\cdot)$ &$\max(0, -\log|\cdot|_v)$. 
\end{longtable}


For $P =[x_0, \dots, x_n] \in \mathbb{P}^n(K)$, define the logarithmic height
$$h_{\mathbb{P}^n}(P) = \frac{1}{[K:\mathbb{Q}]} \sum_{v\in M(K)} n_v\max\left(\log|x_0|_v, \dots, \log|x_n|_v\right) . $$

Suppose $f: \mathbb{P}^n \rightarrow \mathbb{P}^n$ is an endomorphism of degree $d\ge 2$. 
Then following a construction of Tate, Call and Silverman defined in \cite{CS93}
the canonical height $h_f$ associated with $f$ as 
$$h_f(P) = \lim_{n\rightarrow \infty} \frac{h\left(f^{\circ n}(P)\right)}{d^n}.$$

The canonical height satisfies the following properties: 
\begin{itemize}
 \item $\hat{h}_f(P) = h_{\mathbb{P}^n}(P) + O(1)$, 
   \item $\hat{h}_f(P) = d \cdot \hat{h}_f(P)$. 
\end{itemize}
See also Section 3.3 of \cite{Sil07} for more details. 

Now we introduce some notions in algebraic geometry.  
For more information one may refer to \cite{Har77}. 

\begin{definition}
 Let $R = \bar{K}[X_0, \dots, X_n]$ and let $T \subseteq R$ be a set of homogeneous polynomials in $X_0, \dots, X_n$. Every set 
$$Z(T) := \{ P \in \mathbb{P}^n(\bar{K})~|~f(P) = 0 ~\text{for all}~f\in T\}$$
is called a \emph{Zariski closed} subset of $\mathbb{P}^n(\bar{K})$. 
A Zariski closed subset $V\subseteq \mathbb{P}^n(\bar{K})$ is called a \emph{projective variety} if it cannot be written as a union of two Zariski closed proper subsets. 
\end{definition}

To give more general definition of height functions, we need the notion of divisors on nonsingular varieties. See Sections 1.5 and 2.6 of \cite{Har77} for more details. 

\begin{definition}
Let $X$ be a nonsingular projective variety. The group of \emph{Weil divisors} on $X$ is the free abelian group generated by the closed subvarieties of codimension one on $X$. It is denoted by $\mathrm{Div}(X)$. Denote by $K(X)^*$ the multiplicative group of nonzero rational functions on $X$. Each rational function $f\in K(X)^*$ 
gives a \emph{principal divisor} 
$$\mathrm{div}(f) = \sum_{Y\subsetneq X~\mathrm{codimension~1}} \mathrm{ord}_Y(f) \cdot Y.$$
The group $\mathrm{Div}(X)$ divided by the subgroup of principal divisors is called the \emph{divisor class group} of $X$. 
\end{definition}

\begin{remark}
In the case when $X$ is nonsingular, 
the class group is isomorphic to the group $\mathrm{Pic}(X)$. 
For the definition of the latter, see Section 2.6 of \cite{Har77}. 
\end{remark}

\begin{definition}
Suppose $D \in \mathrm{Div}(X)$. The \emph{complete linear system} of $D$ is the set 
$$L(D) = \{f\in K(X)^*~|~D + \mathrm{div}(f) \ge 0\} \cup \{0\}.$$
If $L(D) \neq 0$, then $L(D)$ induces a rational morphism 
$\phi_D: X \dashrightarrow \mathbb{P}^n$. 
For more details, refer to Section A.3 of \cite{HS00}. 
\end{definition}

\begin{definition}
A divisor $D\in \mathrm{Div}(X)$ is said to be \emph{very ample }if the above map $\phi_D$ is an embedding. $D$ is said to be \emph{ample} if an integral multiple $nD$ is very ample. 
\end{definition}

Fix a nonsingular variety $X$ defined over $K$. For each divisor $D\in \mathrm{Div}(X)$ defined over $K$ we can define height functions $h_{X, D}: X(\bar{K}) \rightarrow \mathbb{R}$ as below. 
For more details, including the well-definedness of those height functions, refer to  \cite{HS00}, Theorem B.3.2. 

\begin{itemize}
\item If $D$ is very ample, choose an embedding $\phi_D: X\rightarrow \mathbb{P}^n$. Then define $h_{X,D}(x) = h_{\mathbb{P}^n}(\phi_D(x))$. 
\item If $D$ is ample, then suppose $nD$ is very ample, define $h_{D} = 1/n\cdot h_{nD}$. 
\item  In general, write $D = D_1 - D_2$ with $D_1, D_2$ ample, and define $h_{X, D} = h_{X, D_1} - h_{X, D_2}$. 
\end{itemize}

The following theorem is one of the most important results in Diophantine geometry. 
See also Sections 2.3 and 2.4 of \cite{BG06} and Chapter 4 of \cite{Lang83}.

\begin{theorem}[The Weil Height Machine,  Part of \cite{HS00},Theorem B.3.2]\label{Weil}
In the context of the above paragraphs, the height functions constructed in this way, are determined, up to $O(1)$. They satisfy the following properties. 
\begin{itemize}
\item Let $\phi: X\rightarrow W$ be a morphism and let $D\in \mathrm{Div}(W)$. Then 
$$ h_{X, \phi^* D}(P) = h_{W,D}\left(\phi(P)\right) + O(1)$$
for all $P\in V(\bar{K})$. 
\item Let $D, E\in \mathrm{Div}(X)$. Then 
$h_{X, D+E}  = h_{X, D} + h_{X, E} + O(1)$. 
\item (Northcott's Theorem) Let $D\in \mathrm{Div}(X)$ be ample. 
Then for every finite extension $K'/K$ and every constant $B$, the set 
$$ \{P\in X(K')~|~h_{X, D}(P) \le B  \} $$
is finite. 
\item 
Let $D, E\in \mathrm{Div}(X)$ with $D = E + \mathrm{div}(f)$. 
Then $$h_{X,D}(P)  = h_{X,E}(P) + O(1)$$ for all
$P\in X(\bar{K})$. 
\end{itemize}

\end{theorem}

We will use the following version of Vojta's Conjecture. It is Conjecture 3.4.3 of the monograph \cite{Voj87}.
For the definition of normal crossing divisor, see Chapter 5, Remark 3.8.1 of \cite{Har77}. 

\begin{conjecture}[Vojta]\label{Vojta}
   Let $K$ be a number field, and let $X$ be a nonsingular projective variety defined over $K$.
   Suppose $A$ is an ample normal crossing divisor on $X$ and $K_X$ is the canonical divisor of $X$, both defined over $K$.
   Let $h_A$ and $h_{K_X}$ be the corresponding height functions respectively.
   For each fixed $\varepsilon >0$, there is a Zariski closed proper subset $V$ of $X$
   and a constant $C$ such that
   $$  h_{K_X}(x) \le \varepsilon \cdot h_A(x)+ C$$
   for all $x\in X(K)\setminus V(K)$.
\end{conjecture}

We briefly recall Silverman's idea.
For all $v\in M({\mathbb{Q}})$ and $a\in \mathbb{Z}$, let $v^+(a) = \max(-\log |a|_v, 0)\in [0, +\infty]$.
Silverman began his discussion in \cite{Sil87} by writing the greatest common divisor as
\begin{equation}\label{v+}
\log{\mathrm{gcd}}(a,b) = \sum_{v\in M({\mathbb{Q}})} \min(v^+(a),v^+(b))
\end{equation}
for $a,b\in \mathbb{Z}$. Then he extends this function for $a,b\in\mathbb{Q}$ by the same formula.
Using the ideas from \cite{Sil87},
Silverman observed that the above quantity can be interpreted as a height function with respect to some subschemes,
and furthermore as a height function associated with a divisor on some blown-up surface.
In fact, for algebraic variety $X$, Silverman defined in \cite{Sil87} a height function $h_{X,Y}$ with respect to any closed subschemes $Y$.
These generalized height functions also satisfy certain functorial property.

We need the notion of blowup to interpret gcd in terms of height functions. See pp. 163 of \cite{Har77} for the definition of blowup and strict transform. 
See pp. 28-29 of \cite{Har77} for concrete example of blowing up a point.

\begin{proposition}[\cite{Har77}, Chapter 5, Proposition 3.1]\label{blowingup}
Let $\pi: \tilde{W}\rightarrow W$ be the blowup of a nonsingular surface $W$ at a point $P$. Then 
\begin{enumerate}
\item $\pi$ induces an isomorphism of $\tilde{W}-\pi^{-1}(P)$ and $W-P$, 
\item The set $E: = \phi^{-1}(P)$ is isomorphic to $\mathbb{P}^{1}$. It is called the \emph{exceptional divisor} of the blowup $\pi$, 
\item $\tilde W$ is nonsingular. 
\end{enumerate}
\end{proposition}

The following definition is a slight generalization of that given by Silverman in \cite{Sil05}.

\begin{definition}
Let $K$ be a number field and let $X/K$ be a smooth variety. Let $Y/K \subsetneq X/K$ be a
subscheme of codimension $r \ge 2$. Let $ \pi:\tilde{X}\rightarrow X$ be the blowup of $X$
along $Y$, and let $\tilde{Y} = \pi^{-1} (Y ) $ be the exceptional divisor of the blowup.
For $x\in (X - Y)(K)$, we let
$\tilde{x} = \pi^{-1} (x) \in \tilde{X}$.
The generalized (logarithmic) greatest common divisor of the point
$x\in (X - Y)(k)$ with respect to $Y$ is the quantity
$$h_{\gcd} (x;Y) := h_{X,Y}(x) = h_{\tilde{X}, \tilde{Y}} (\tilde{x}) $$
where the last inequality follows from the Weil Height Machine, 
as generalized by Silverman in \cite{Sil05}. 
\end{definition}


For a number fields $K$ and for $a,b\in K$ we also define the generalized gcd as
\begin{equation}\label{ggcd}
h_{\gcd}(a,b) = \frac{1}{[K:\mathbb{Q}]}\sum_{v\in M(K)} n_v \min ( v^+(a), v^+(b)).
\end{equation}
We also define
\begin{equation}\label{ggcd2}
h_{\gcd, \text{fin}}(a,b) = \frac{1}{[K:\mathbb{Q}]}\sum_{v\in M(K), \text{fin}} n_v \min ( v^+(a), v^+(b)).
\end{equation}
Then clearly $h_{\gcd, \text{fin}} \le h_{\gcd}$.

As a consequence of the Weil height machine, the relationship between these two $h_{\gcd}$ is shown at the end of this paragraph.
See \cite{Sil87} and \cite{Sil05} for some interesting cases over $\mathbb{Z}$ where the contribution from the places at infinity is zero or bounded.
 Suppose $K$ is a number field. Let $X = \airplane$ and let
   $f(X_1)\in K[X_1], g(X_2)\in K[X_2]$ be polynomials. 
Then over $\overline{\Q}$ the vanishing set $Z(f)$ and $Z(g)$ define two divisors $D_1$ and $D_2$ on $X$.
Set $Y = D_1 \cap D_2$.
   Then for all points $x = (x_1,  x_2)\in \airplane$ with $x_1,x_2\in K$,
such that $f(x_1) \neq 0$ and $g(x_2) \neq 0$,
we have
   \begin{equation*}
   \begin{aligned}
     h_{\gcd}\left(f(x_1),g(x_2)  \right) &= h_{\airplaneOne, (0,0)}(f(x_1), g(x_2)) \\
      &= h_{\airplaneOne, (f,g)^{*}(0,0)}(x_1, x_2) \\
        & = h_{\mathrm{gcd}}(x;Y) + O(1),
   \end{aligned}
   \end{equation*}
   where the second equality follows from Theorem 2.1(h) of \cite{Sil87}.

Our goal is to prove the following theorem. 

\begin{theorem}\label{HGCD}
  Assume Vojta's Conjecture (Conjecture \ref{Vojta}).
   Let $K$ be a number field. 
   Suppose $a,b,\alpha,\beta\in K$. Let $f,g\in K(X)$ with degrees $\deg f= \deg g=:d\ge 2$.
   Assume that the sequence
   $(f^{\circ n}(a) - \alpha, g^{\circ n}(b) - \beta)_n
     \subseteq \mathbb{P}^1(\overline{\mathbb{Q}})\times \mathbb{P}^1(\overline{\mathbb{Q}})$ is generic,
   and $\alpha$ and $\beta$ are not exceptional for $f$ and $g$ repsectively.
   Then for each given $\varepsilon >0$, there exists 
   a constant $C =  C(\varepsilon,a,b,\alpha,\beta,f,g)$ such that
   for all $n\ge 1$, we have
   $$h_{\mathrm{gcd}}(f^{\circ n}(a)-\alpha, g^{\circ n}(b) -\beta)
   \le {\varepsilon\cdot d^n} + C. $$
\end{theorem}

We can also conclude the periodicity of an irreducible component of the Zariski closure $\overline{(f^n(a), g^n(b))_n}$
under $(f,g)$ in the cases when the
Dynamical Mordell-Lang Conjecture is proved. See Section \ref{DML}.

Thanks to the powerful theorems proved in \cite{BD13}, \cite{MS09}, and \cite{Pak15},
we can give some concrete conditions for $(f^n(a),g^n(b))_n$ being generic in the case when $f=g$ are so-called non-special polynomials
(See Section \ref{DML}).

\begin{theorem}\label{BD}
Assume Vojta's Conjecture (Conjecture \ref{Vojta}).
Let $K$ be a number field and $f\in K[x]$ be a polynomial of degree $d\ge 2$. Assume that $f$ is not conjugate
(by a rational automorphism defined over $\overline{K}$) to a power map or a Chebyshev map.
Suppose $a,b,\alpha, \beta\in K$ and $\alpha,\beta$ are not exceptional for $f$.
Assume that there is no polynomial $h\in \overline{K}[x]$ such that $h\circ f^{\circ k} = f^{\circ k}\circ h$ for some $k\in \mathbb{N}_{>0} $
and $h(a) = b,~h(\alpha) = beta$ or $h(b) =a,~h(\beta) = \alpha$ for some $m\in\mathbb{N}$,
then for any $\varepsilon >0$, there exists a $C  =  C(\varepsilon,a,b,\alpha,\beta,f,g) >0$ such that for all $n\ge 1$, we have
$$h_{\gcd}(f^{\circ n}(a) - \alpha, f^{\circ n}(b) - \beta) \le \varepsilon\cdot  d^n + C.$$
\end{theorem}

\section{The Proof of Theorem \ref{HGCD}}

Throughout this section we donte by $X$ the surface $\mathbb{P}^1 \times \mathbb{P}^1$.

\subsection{Algebraic Geometry of $\mathbb{P}^1 \times \mathbb{P}^1$ and its Blowups}\label{AlgGeo}
\label{AGSection}

By Chapter 2, Example 6.6.1 of \cite{Har77} we have 

   \begin{equation*}
      \begin{aligned}
       \mathrm{Pic}(X) \cong \mathbb{Z} \oplus \mathbb{Z}. 
      \end{aligned}
   \end{equation*}
where the image of an irreducible curve $C$ is the degrees of its projection into the two coordinates 
$(\deg(\mathrm{pr}_1: C \rightarrow \mathbb{P}^1), \deg(\mathrm{pr}_2: C\rightarrow \mathbb{P}^1))$. 
More generally, if the image of a divisor $D\in \mathrm{div}(X)$ is $(a,b)$, then we say that $D$ is 
\emph{of type $(a,b)$}.

Let $K$ be a number field.
Suppose
$f \in K[X_1]$ and $g\in K[X_2]$ are square-free polynomials in one variable, 
Let $Y$ be the scheme-theoretic intersection
$$Y = Z(f) \cap Z(g)\subseteq \airplane, $$
which is the subscheme defined by the ideal $(f) + (g)$, 
is then a reduced cycle of codimension 2. 

   Suppose $Z(f) = \{ \alpha_1, \dots, \alpha_m\}$, $Z(g) = \{\beta_1, \dots, \beta_n\}$.
   Then $Y = \cup_{1\le i\le m, 1\le j \le n}\{ (\alpha_i, \beta_j)\}$, each with multiplicity one.
   Also divisors $\{X_1 = \alpha_i\}$ and $\{X_2 = \beta_j\}$ meet transversally, hence $Y$ is a reduced cycle of codimension 2.
   To simplify notations write $Y= \{Q_1, \dots, Q_s \}$.
   Let $\pi: \tilde{X}\rightarrow X$ be the blowup of $X =\airplaneOne$ along $Y$, let $\tilde{Y}$ be the preimage of $Y$, and let $\tilde{P}$ be the preimage of $P$.
   Then $\tilde{X}$ is a nonsingular variety by Proposition \ref{blowingup}.

The following properties are useful to determine find 
the canonical divisor and an ample divisor on $\tilde{X}$. 

\begin{proposition}\label{blown}[\cite{Har77}, Chapter 5, Propositions 3.2 and 3.3]
Suppose $\pi: \tilde{X} \rightarrow X$ is the blowup of 
a surface $X$ at a point $P$. Then there is a canonical isomorphism $\mathrm{Pic}(\tilde{X}) \cong \mathrm{Pic}(X) \oplus \mathbb{Z}$. The intersection theory on $\tilde{X}$ is determined by the rules:
\begin{enumerate}
\item if $C,D\in \mathrm{Pic}(X)$, 
then $(\pi^*C.\pi^*D) = (C. D)$, 
\item if $C\in \mathrm{Pic}(X)$, then $(\pi^*C.E) = 0$, 
\item it holds that $E^2 = -1$, 
\item if $C\in \mathrm{Pic}(X)$ and $D\in \mathrm{Pic}(\tilde{X})$, then 
$(\pi^*C. D) = (C. \pi_* D)$; 
\end{enumerate}
else, the canonical divisor of $\tilde{X}$ is given by $K_{\tilde{X}} = \pi^*K_X + E$ where $E$ is the exceptional divisor. 
\end{proposition}

Since the blowup of $\tilde{Y}$ does not involve blowup at a point on an exceptional curve, repeated use of Proposition \ref{blown} yield $$\mathrm{Pic}(\tilde{X}) =  \pi^*\mathrm{Pic}({X})\bigoplus\bigoplus_{i=1}^s \mathbb{Z}\cdot \tilde{Y}_i. $$

In addition, 
   \begin{equation*}
      \begin{aligned}
          K_{\tilde{X}} = \pi^*K_X + \tilde{Y}_{1} + \dots + \tilde{Y}_{s}
      \end{aligned}
   \end{equation*}
   where each $\tilde{Y}_{i}$ is the preimage of $Q_i$.
   We can choose $-K_X$ to be the normal crossing divisor
   $\{X_1 = a \} + \{X_1 = b\} + \{X_2 = a'\} + \{ X_2 = b'\}$ where $a,b,a',b'$ are distinct nonzero algebraic numbers in $K$. 
   By Definition 2 of \cite{Sil05}, we still have $h_{\mathrm{gcd}} (P;Y ) = h_{\tilde{X}, \tilde{Y}}(\tilde{P}).$

   To apply Vojta's Conjecture, let $A\in \mathrm{Div}(X)$ be a divisor of type $(1,1)$ and consider the $\mathbb{Q}$-divisor
   \begin{equation*}
      \begin{aligned}
          {\tilde{A}}: = \pi^*A - \frac{1}{N}\left( \tilde{Y}_1 + \dots + \tilde{Y}_s\right)\in \mathrm{Div}(\tilde{X})\otimes \mathbb{Q}.
      \end{aligned}
   \end{equation*}

\begin{lemma} \label{AG}
$\tilde{A}$ is ample when $N>s$.
\end{lemma}

\begin{proof}
We need the following definition from Chapter 1, Exercise 5.3 of \cite{Har77}. 
\begin{definition}
Let $Y \subseteq \mathbb{A}^2$ be a curve defined by the equation $f(X_1,X_2) = 0$. Let
$P = (x_1,x_2)$ be a point of $\mathbb{A}^2$. 
Make a linear change of coordinates so that $P$ becomes the point $(0,0)$. 
Then write $f$ as a sum $f = f_0 + f_1 + ... + f_d$, where
$f_i$ is a homogeneous polynomial of degree $i$ in $X_1$ and $X_2$. Then we define the \emph{multiplicity} of $P$ on $Y$, denoted $\mu_P(Y)$, to be the least $r$ such that $f_r\neq 0$.
\end{definition}

We also need the following lemma, which we won't prove. 
 \begin{lemma}[\cite{Har77}, Chapter 1, Exercise 7.5(a)]\label{mul}
    An irreducible curve $Y$ of degree  $d  >  1$  in  $\mathbb{P}^2$  cannot have a point of multiplicity $\ge d$.
 \end{lemma}
Now let $C\subseteq \airplaneOne$ be an irreducible curve of type $(a,b) $.
Let $\tilde{C}$ be its strict transform.
By Lemma \ref{mul} we know that
$C$ cannot have a point of multiplicity $\ge\deg(C)$.
By Proposition 3.6 of \cite{Har77}, 
$(\tilde{Y}_i.\tilde{C}) = (\tilde{Y}_i.\pi^*C-\mu_{Q_i}(C)\cdot \tilde{Y}_i) = \mu_{Q_i}(C)$. 
Now let $\mathrm{pr}_i: C\rightarrow \mathbb{P}^1$ be the projection to the $i$-th coordinate.
Then $\deg \mathrm{pr}_1 = a,~\deg \mathrm{pr}_2 = b$. 
This is to say, if we restrict $C$ to $\mathbb{A}^2$, 
then the defining equation has degree $b$ on $X_1$ and degree $a$ on $X_2$. 
It follows that $\deg(C) \le a +b$.
By the \text{projection formula}, we have
$$\left(\pi^*A. \tilde{C}\right) =  \left(A.\pi_*\tilde{C}\right)=  \left(A. C\right) = a+b.$$
Then
 \begin{equation*}
   \begin{aligned}
       (\tilde{A}. \tilde{C})
       & = \left(\pi^*A. \tilde{C}\right)  - \frac{1}{N}\left( (\tilde{Y}_1.\tilde{C}) + \dots + (\tilde{Y}_s.\tilde{C})\right) \\
                           &=  a + b - \frac{1}{N}\left( \mu_{Q_1}(C) + \dots + \mu_{Q_s}(C)\right) \\
                           &\ge  a +b - \frac{1}{N} \cdot s \cdot (a+b)  \\
                           & > 0
   \end{aligned}
 \end{equation*}
as $N> s $.
Since $Y_i$'s are preimages of distinct $Q_i$'s, so $(\tilde{Y}_i, \tilde{Y}_j) = -\delta_{ij}$ and
 \begin{equation*}
    \begin{aligned}
       (\tilde{A}. \tilde{Y}_i) &=  \left(\pi^*A. \tilde{Y}_i\right)  -
                           \frac{1}{N}\left( (\tilde{Y}_1.\tilde{Y}_i) + \dots + (\tilde{Y}_s.\tilde{Y}_i)\right)   \\
                                  &= 0 - \frac{1}{N} \left( -\delta_{1i} - \dots - \delta_{si}\right)  \\
                                  &= \frac{1}{N}.
    \end{aligned}
 \end{equation*}
Finally by the previous equality
\begin{equation*}
\begin{aligned}
(\tilde{A}. \tilde{A})  &=
\left(\tilde{A}.~ \pi^*A\right) - \frac{1}{N}\left((\tilde{A}.~  \tilde{Y}_1) + \dots + (\tilde{A}.~ \tilde{Y}_s)\right)  \\
& > \left(\pi_{*}\tilde{A}.~A\right) -\frac{1}{N}\cdot \frac{s}{N} \\
&= (A.A) - \frac{s}{N^2} \\
&\ge 1 + 1 - \frac{s}{N^2} \\
&>0
\end{aligned}
\end{equation*}
as $N >s$.
 But $$\mathrm{Pic}(\tilde{X}) =  \pi^*\mathrm{Pic}({X})\bigoplus\bigoplus_{i=1}^s \mathbb{Z}\cdot\tilde{Y}_i, $$
 and every effective curve $C$ in $\tilde{X}$ is linearly equivalent to a nonnegative combination of $\tilde{Y}_i$'s and the strict transform of
effective curves in $X$,
so $\tilde{A}$ is ample by the Nakai-Moishezon criterion (see Chapter 5, Theorem 1.10 of \cite{Har77}).
\end{proof}

\subsection{The Proof, Continued}
We first prove the following modification of Theorem 2 of Silverman (\cite{Sil05}).
Recall that a one-variable polynomial over a field $K$ is called {\it square free} if it does not have repeated roots in $\overline{K}$.


\begin{theorem}\label{Sil}
Let $K$ be a number field.
Suppose
$f \in K[X_1]$ and $g\in K[X_2]$ are square-free polynomials in one variable, 
Let
$$Y = Z(f) \cap Z(g)\subseteq \airplane $$
as in the Subsection \ref{AlgGeo}. 

Assume that Vojta's conjecture is true (for $\airplaneOne$ blown up along $Y$).
Fix $\varepsilon > 0$.
Then there is a algebraic subset $V\subsetneq \airplaneOne$,
depending on $f,g$ and $\varepsilon$,
so that every $P = (x_1,x_2) \in \mathbb{P}^1(K)\times  \mathbb{P}^1(K)$ satisfies either
\begin{enumerate}
   \item $P\in V$, or
   \item   $h_{\gcd} ( f(x_1),g (x_2) )\le
(3 + \varepsilon) \left( h(x_1) +h( x_2) \right) +  O(1).$
\end{enumerate}
\end{theorem}

\begin{proof}[Proof of Theorem \ref{Sil}]
   Use the notations in Section \ref{AGSection}. We follow the proof in \cite{Sil05}. 
By Lemma \ref{AG} and assuming Vojta's Conjecture we have
 \begin{equation*}
   \begin{aligned}
    h_{\tilde{X}, K_{\tilde{X}}}(\tilde{P}) \le \varepsilon \cdot h_{\tilde{X}, {\tilde{A}}}(\tilde{P}) + C_\varepsilon
   \end{aligned}
 \end{equation*}
for all $P\in X(K)\setminus V(K)$. 
Also $ K_{\tilde{X}} = \pi^* K_X + \tilde{Y}$ and $\tilde{A} = \pi^* A - 1/N\cdot \tilde{Y}$, so
\begin{equation*}
   \begin{aligned}
      h_{\tilde{X}, \pi^* K_{{X}}}({\tilde{P}})  +h_{\tilde{X}, \tilde{Y}}(\tilde{P})
           &\le \varepsilon \cdot h_{\tilde{X}, {\pi^*{A}}}(\tilde{P}) - \frac{1}{N} \cdot h_{\tilde{X}, \tilde{Y}}(\tilde{P}) + C_\varepsilon, \\
          h_{{X}, K_{{X}}}({{P}})  +   \left( 1 + \frac{1}{N}\right)h_{\tilde{X}, \tilde{Y}}({P})
                           &\le \varepsilon \cdot h_{{X}, {A}}({P}) + C_\varepsilon',     \\
     \left( 1 + \frac{1}{N}\right)h_{\mathrm{gcd}}(P;Y) &\le \varepsilon\cdot h_{X, A}(P) + h_{X,-K_X}  (P) + C_\varepsilon', \\
       h_{\mathrm{gcd}}(P;Y) &\le \varepsilon\cdot h_{X, A}(P)  +   h_{X,-K_X}  (P) + C_\varepsilon''. \\
   \end{aligned}
\end{equation*}
%
But $K_X$ is linearly equivalent to $-2A$, and let $P= (x_1,x_2)$. Then
\begin{equation*}
   \begin{aligned}
     h_{X, -K_X}(P) &= 2\cdot \left(h(x_1) + h(x_2)\right) + O(1), \\
      h_{X,A}(P) & = h(x_1) + h(x_2) , \\
      h_{\gcd}(P;Y) &= h_{\gcd}(f(x_1),g(x_2)). 
   \end{aligned}
\end{equation*}
Now Theorem \ref{Sil} is verified.
\end{proof}

\newcommand{\rad}{\mathrm{rad}}

\begin{proof}[Proof of Theorem \ref{HGCD}]
We begin with the following

\begin{lemma}
   Let $\sigma, \tau\in K(x)$ be M\"obius transformations defined over $K$. Set $f_\sigma = \sigma f\sigma^{-1},~g_\tau =  \tau  g\tau^{-1}$.
   Then there exists a constant $C>0$, depending on $\alpha, \beta, f, g, \sigma, \tau$,
such that for all $a,b\in K$, and for all $n\in \mathbb{N}$, we have
 \begin{equation*}
\left|h_{\gcd, \mathrm{fin}}\left(f_\sigma^{\circ n}\left(\sigma a\right) - \sigma\alpha,
                                    g_\tau^{\circ n}\left(\tau b\right)  - \tau\beta\right)
         - h_{\gcd,\mathrm{fin}}(f^{\circ n}(a) - \alpha, g^{\circ n}(b) - \beta)\right|  \le C.
\end{equation*}
\end{lemma}

\begin{proof}
   It suffices to show that for any fixed $\alpha \in K$, and for any fixed M\"obius transformation $\sigma$,
   there exists a finite set $S\subset M(K)_{\text{fin}}$ and a constant $C'>0$, such that for all $x\in K$ and $v\in S$, we have
$\left|  v^{+}\left(\sigma x - \sigma \alpha\right)
         - v^+(x - \alpha)\right|  \le C'$,
and for all $x\in K$ and $ v\in M(K)_{\text{fin}}\setminus S$, we have $  v^{+}\left(\sigma x - \sigma \alpha\right) =  v^+(x - \alpha)$.

 Since each M\"obius transformation defined over $K$ is a composition of translations, dilations and inverses defined over $K$,
 it suffices to prove the result for the case when $\sigma$ is one of the above three types of maps.
 The result is trivial for translations and dilations.

 If $\sigma(x) = 1/x$, write $x = x_1 /x_2, \alpha = \alpha_1/\alpha_2$, $x_1, x_2,\alpha_1,\alpha_2 \in \mathcal{O}_K$. Since the class number of $K$ is finite,
 there exists $\gamma\in \mathcal{O}_K$ such that for fixed $\alpha\in \mathcal{O}_K$
 and for all $x\in \mathcal{O}_K$ we can always choose $x_1,x_2,\alpha_1,\alpha_2$ such that the ideals
 $\gcd(x_1, x_2)~|~\gamma,~ \gcd(\alpha_1, \alpha_2) ~|~\gamma$.
 Now
 \begin{equation*}
  \left|x-a\right|_v = \left|\frac{\alpha_2 x_1 - \alpha_1 x_2}{\alpha_2 x_2}\right|_v,
  ~\left|\sigma x  - \sigma \alpha \right|_v = \left|\frac{\alpha_2 x_1 - \alpha_1 x_2}{\alpha_1 x_1}\right|_v.
 \end{equation*}
 But the ideal
 \begin{equation*}
    \begin{aligned}
       \gcd({\alpha_2 x_1 - \alpha_1 x_2}, {\alpha_2 x_2})~&|~\gcd({\alpha_2^2 x_1 - \alpha_1\alpha_2 x_2}, {\alpha_1\alpha_2 x_2})
        = \gcd(\alpha_2^2 x_1 , \alpha_1\alpha_2 x_2)\\
        &|~\gcd(\alpha_1\alpha_2^2 x_1 , \alpha_1\alpha_2^2 x_2) ~|~ \alpha_1\alpha_2^2\gamma,
    \end{aligned}
 \end{equation*}
 so
  \begin{equation*}
    \begin{aligned}
    v^+(\alpha_2 x_1 - \alpha_1 x_2) -v(\alpha_1\alpha_2^2\gamma)  \le  v^+(x - \alpha) \le v^+(\alpha_2 x_1 - \alpha_1 x_2).
    \end{aligned}
 \end{equation*}
 Similarly
  \begin{equation*}
    \begin{aligned}
    v^+(\alpha_2 x_1 - \alpha_1 x_2) -v(\alpha_1^2\alpha_2\gamma)  \le  v^+(\sigma x - \sigma\alpha) \le v^+(\alpha_2 x_1 - \alpha_1 x_2).
    \end{aligned}
 \end{equation*}
 Therefore
 \begin{equation*}
    \begin{aligned}
    \left|v^+(\sigma x- \sigma \alpha)  - v^+(x - \alpha) \right|
    \le \max\left(v(\alpha_1\alpha_2^2), v(\alpha_1^2\alpha_2\gamma)\right) \le v(\alpha_1^2 \alpha_2^2\gamma).
    \end{aligned}
 \end{equation*}
Hence we may choose $S = \{v\in M(K)_{\text{fin}}~|~v(\alpha_1)\neq 0,~v(\alpha_2) \neq 0~\text{or}~v(\gamma) \neq 0\}$.
\end{proof}

Therefore for $h_{\gcd,\mathrm{fin}}$ we may assume that $\alpha = \beta= 0$.
For any fixed integer $D$,
write in the lowest terms $f^{\circ D}  = F_1/F_2$ and $g^{\circ D} = G_1/G_2$ where $F_1, F_2,G_1,G_2$ are polynomials with coefficients in $\mathcal{O}_K$.
For the same reason we may also assume that all $D$-th preimages of $0$ under $f$ and $g$ are not $\infty$.

Write
$$F_1(x) = a_0 +  \dots + a_Nx^N,$$
$$F_2(x) = b_0 + \dots + b_Mx^M,$$
$$G_1(x) = a_0'  + \dots + a'_{N'}x^{N'},$$
$$G_2(x) = b_0' + \dots + b'_{M'}x^{M'}$$
with all coefficients in $\mathcal{O}_K$.
Then $N\ge M$.
Let
$$S:=\{\text{non-archimedean place}~v~|~v(a_N) \neq 0,~ v(b_M) \neq 0,~ v(a'_{N'}) \neq 0,~ \text{or}~ v(b'_{M'}) \neq 0\}. $$
Then $S$ is finite. For all non-archimedean place $v\notin S$ and for any $x_0\in K$,
if $v(x_0) \ge 0$, then $v(F_2(x)) \ge 0$ and hence
$v^{+}(f^{\circ D}(x_0)) \le v^{+}\left(F_1 (x_0)\right) . $
If $v(x_0) < 0 $, then
$$   v^{+}(f^{\circ D}(x_0)) = v^{+}\left(\frac{a_N x_0^N}{b_M x_0^M}\right)
=v^{+}\left({ x_0^{N-M}}\right) =0 \le v^{+}\left(F_1 (x_0)\right). $$
In either case we have
\begin{equation*}
   \begin{aligned}
       v^{+}(f^{\circ D}(x_0)) &\le v^{+}\left(F_1 (x_0)\right) .
   \end{aligned}
\end{equation*}
Similarly for any $v\notin S$ and for any $y_0\in K$,
\begin{equation*}
  \begin{aligned}
       v^{+}(g^{\circ D}(y_0)) &\le v^{+}\left(G_1 (y_0)\right) .
  \end{aligned}
\end{equation*}

Therefore the sum of the finite parts of $h_{\gcd}$ outside $S$ satisfy
\begin{equation}\label{denom}
   \begin{aligned}
        h_{\gcd, S}\big(f^{\circ D}(a'), g^{\circ D}(b')\big) & :=
              \frac{1}{[K:\mathbb{Q}]}\sum_{v\in M(K)_{\text{fin}}\setminus S} n_v \min \left( v^+(f^{\circ D}(a')), v^+( g^{\circ D}(b'))\right) \\
               &\le \frac{1}{[K:\mathbb{Q}]}\sum_{v\in M(K)_{\text{fin}}\setminus S} n_v \min \left( v^+(F_1(a')), v^+(  G_1 (b')) \right)\\
               & \le  h_{\gcd, S}\big(F_1(a'),  G_1 (b')\big).
   \end{aligned}
\end{equation}

Let $F_1^{\rad}(x) = \rad (F_1)(x)$, and let $G_1^{\rad}(y) = \rad (G_1)(y)$,
where for a one-variable polynomial $P$, $\rad(P)$ is the product of all monic irreducible polynomials dividing $P$.
As the sequence $\left(f^{\circ (n-D)}(a), g^{\circ (n-D)}(b)\right)_n$ is generic in
$\mathbb{P}^1(\overline{\mathbb{Q}}) \times \mathbb{P}^1(\overline{\mathbb{Q}})$,
there exists $N'' = N''(\varepsilon, f,g, a,b, \alpha, \beta)$, such that for all $n \ge N''$ we have
\begin{equation}\label{condition}
  \left( f^{\circ (n-D)}(a), g^{\circ (n-D)}(b)\right)\notin V(K)
\end{equation}
where $ V$ is as in Theorem \ref{Sil}.
Apply Theorem \ref{Sil} to the point $\left(f^{\circ (n-D)}(a), g^{\circ (n-D)}(b)\right) $ and the functions $F_1^{\rad}$ and $G_1^{\rad}$,
with $\varepsilon = 1$.
Let $u = f^{\circ (n-D)}(a),v = g^{\circ (n-D)}(b)$.
Then
\begin{equation}\label{iter}
   \begin{aligned}
      h_{\gcd}\left(F_1^{\rad}(u), G_1^{\rad}(v)\right)
   &\le 4 \left(h\left(u\right)  + h(v)\right) +O(1).
   \end{aligned}
\end{equation}
Set
\begin{equation*}
   \begin{aligned}
M' &= \max_{f^{\circ D} (x) = \alpha, g^{\circ D}(y) = \beta}
\left( e_{x}(f^{\circ D} -\alpha), e_{y}(g^{\circ D} -\beta)\right)
   \end{aligned}
\end{equation*}
where $e_Q(\phi)$ is the multiplicity of $\phi$ at $Q$.
Combining the above with (\ref{denom}) and (\ref{iter}) we have
   \begin{align*}
&\ \ \ \     h_{\gcd, S}\left(f^{\circ n}(a), g^{\circ n}(b)\right)\\
      &= h_{\gcd, S}\left(f^{\circ D}(f^{\circ (n-D)}(a)), g^{\circ D}(g^{\circ (n-D)}(b))\right)                  \\
               &\le  h_{\gcd,S}\left(F_1 (f^{\circ (n-D)}(a)), G_1(g^{\circ (n-D)}(b)))\right)\\
      &\le  h_{\gcd, S}\left(\left(F_1^{\rad}\circ f^{\circ(n-D)}(a)\right)^{M'}, \left(G_1^{\rad}\circ g^{\circ(n-D)}(b)\right)^{M'} \right) +O(1) \\
     &\le 
        M'\cdot \left( {4} \cdot{h}\left(f^{\circ(n-D)}(a)\right) + 4\cdot h\left( g^{\circ(n-D)}(b)\right) +O(1) \right) +O(1)  ~(\text{by (\ref{iter})})\\
     &\le M '\cdot\left( {4}d^{n-D} \cdot\hat{h}_f \left(a\right)  + {4} d^{n-D}
     \cdot \hat{h}_g \left(b\right) + O(1) \right)+O(1)\\
    & \le  d^n\cdot \frac{M'} {d^{D}}
      \cdot \left({4} \hat{h}_f(a) + 4\hat{h}_g(b) + C\right) + O(1).
   \end{align*}
   Since $\alpha,\beta$ are not exceptional for $f,g$ respectively,
by the proof of Lemma 3.52 of \cite{Sil07},
we can choose $D = D(\varepsilon, f, g,a,b)\in \mathbb{N}$ sufficiently large so that
$$
  \frac{M'}{d^D}\cdot \left({4} \hat{h}_f(a) + 4\hat{h}_g(b) + C \right)
 < \frac{\varepsilon}{2}. $$
Thus, we have
\begin{equation*}
 h_{\gcd, S}\left(f^{\circ n}(a), g^{\circ n}(b)\right)
           \le\frac{\varepsilon}{2}\cdot d^n + O(1). 
 \end{equation*}
 Hence in old coordinate we have
\begin{equation}\label{nonarch}
 h_{\gcd, S}\left(f^{\circ n}(a)-\alpha, g^{\circ n}(b)-\beta\right)
           \le\frac{\varepsilon}{2}\cdot d^n + O(1). 
 \end{equation}

For any $v\in S$ or any infinite $v$, we use the old coordinate and we have
\begin{equation}\label{bydef}
   \begin{aligned}
    \hspace{2em} & \min \left( v^+(f^{\circ n}(a)-\alpha ), v^+( g^{\circ n}(b)-\beta )\right) \\
=& \min\Big(\max(-\log|f^{\circ n}(a)-\alpha |_v,0), \max(-\log| g^{\circ n}(b)-\beta|_v,0) \Big).
   \end{aligned}
\end{equation}
Since $0$ is not exceptional with respect to $f$ and to $g$, by Theorem E of \cite{Sil93}, we know for all sufficiently large $n\in \mathbb{N}$,
\begin{equation}\label{arch}
   \begin{aligned}
      -\log|f^{\circ n}(a) -\alpha|_v \le \frac{\varepsilon}{2\cdot ([K:\Q]+|S|)}\cdot d^n + O(1),\\
      ~ -\log|g^{\circ n}(b)-\beta|_v \le \frac{\varepsilon}{2\cdot ([K:\Q] +|S|)} \cdot d^n+ O(1).
   \end{aligned}
\end{equation}
Combining equations (\ref{nonarch}), (\ref{bydef}) and (\ref{arch}), we obtain the requested estimate.
\end{proof}

\section{On the genericity condition}\label{DML}

The Dynamical Mordell-Lang Conjecture predicts that given an endomorphism $\phi: X\rightarrow X$ of a complex quasi-projective variety $X$,
for any point $P\in X$ and any subvariety $Y\subsetneq X$,
the set $\{n\in\mathbb{N}~|~\phi^{\circ n}(P) \in Y \}$ is a finite union of arithmetic progressions (sets of the form
$\{a, a+d, a+ 2d,\dots \} $ with $a,d\in \mathbb{N}_{\ge 0})$.
The Dynamical Mordell-Lang Conjecture was proposed in \cite{GT09}.
See also \cite{Bel06} and \cite{Den92} for earlier works.
In the case of \'etale maps we know that the Dynamical Mordell-Lang Conjecture is true.
See the recent monograph \cite{BGT16}. Xie proved in \cite{Xie15} the Dynamical Mordell-Lang Conjecture for polynomial endomorphisms of the affine plane.

\begin{proof}[Proof of Theorem \ref{BD}]
The result is clearly true in the case when $(a,b)$ is preperiodic under $(f,f)$.
When $(a,b)$ is not preperiodic under $(f,f)$,
by Theorem \ref{GCD} it suffices to show that the sequence $(f^{\circ n}(a), f^{\circ n}(b))_n$ is generic.
If there were infinitely many iterates $(f^{\circ n}(a), f^{\circ n}(b))$ lying on a curve $C$,
then by Theorem 0.1 of \cite{Xie15}, the Dynamical Mordell-Lang Conjecture for polynomial endomorphisms of the affine plane, 
we know $C$ itself is periodic under $(f,f)$.
Replacing $f$ by an iterate $f^{\circ m}$ we may assume that $C$ is fixed under $(f,f)$.
Now we can apply the results of \cite{Pak15} and \cite{MS09} classification for invariant curves.
In fact, using these results
Baker and DeMarco demonstrated in Page 32 of \cite{BD13} that the irreducible invariant curve in the above theorem must be a graph
of the form $ y = h(x)$ or $x = h(y)$, for a polynomial $h$ which commutes with some $f^{\circ k}$ with initial conditions as in Theorem \ref{BD}.
This contradicts the assumption of Theorem \ref{BD}.
\end{proof}


We give two examples to show that if the assumption of Theorem \ref{BD} is not verified,
then we might not have the upper bound.

\begin{example}
Under the hypothesis of the above proof and use the same notations.
Assume that the curve is given by $y = h(x)$ and
$h\circ f^{\circ k} = f^{\circ k}\circ h$ for some $k\in \mathbb{N}_{>0} $. Suppose $n = mk$ with $k\in \mathbb{N}$. If $h(\alpha) = \alpha$, then
\begin{equation*}
   \begin{aligned}
 \gcd (f^{\circ n}(a) - \alpha, f^{\circ n}(b) - \alpha) & = \gcd(f^{\circ mk}(a) - \alpha, f^{\circ mk}\left(h(a)\right) - \alpha)\\
 &= \gcd(f^{\circ mk}(a) - \alpha, h(f^{\circ mk}(a)) - h(\alpha))\\
 &= |f^{\circ mk}(a) - \alpha| = |f^{\circ n}(a) - \alpha|.
   \end{aligned}
\end{equation*}
\end{example}

\begin{example}
   Let $f(x) = g(x) = x^3 + x$. Assume $a = -b$ and $\alpha = -\beta$. Then for $h(x) = -x$,
   we have $h\circ f = f\circ h$, $h(a) = b$ and $h(\alpha)   = \beta$. Now
   $$f^{\circ n}(a) -\alpha = f^{\circ n}(-b) + \beta = -f^{\circ n}(b) + \beta = -(g^{\circ n}(b) - \beta), $$
   so
   $$\gcd(f^{\circ n}(a) -\alpha, g^{\circ n}(b) -\beta) = |f^{\circ n}(a)  - \alpha|  \gg |a|^{\delta^n}$$
   for any $\delta<3$.
\end{example}

In the case of power maps, if $(f^{\circ n}(a), g^{\circ n}(b))_n$ is generic,
the following unconditional result is proved by Corvaja and Zannier (\cite{CZ05}).

\begin{example}\label{power}

Suppose $K$ is a number field and suppose $a,b,\alpha, \beta\in K$.
Also suppose that $f$ and $g$ are power maps, and $a,b$ are multiplicatively independent.
Let $d = \max(\deg f, \deg g)$,
then for each fixed $\varepsilon >0$, there exists some $C = C(f,g,a,b)$
such that
\begin{equation}\label{CZ}
   \begin{aligned}
      \gcd(f^{\circ n}(a) - \alpha, g^{\circ n}(b) - \beta)
        &\le  C\cdot\max\left( {h}(a), {h}(b) \right)^{\varepsilon d^n}.
   \end{aligned}
\end{equation}
In fact, the genericity of the sequence $(f^{\circ n}(a), g^{\circ n}(b))_n$ is equivalent to the multiplicative independence of $a$ and $b$.
The assumption that $\alpha$ and $\beta$ are not exceptional implies that $\alpha \neq 0$ and $\beta \neq 0$.
Then Inequality (\ref{CZ}) is a consequence of Inequality (1.2) of Corvaja and Zannier (\cite{CZ05}).
\end{example}

Now we provide an example to explain that the genericity of $ (f^{\circ n}(a), f^{\circ n}(b))_n$ is necessary for power maps.
\begin{example}
   Let $ a = 125, b =25, \alpha = \beta = 1, f(x) = x^2, g(y) = y^2$. Then
$\gcd(f^{\circ n}(a) - \alpha, g^{\circ n}(b) - \beta)  $ is divisible by $5^{2^n} - 1 = O\left((f^{\circ n}(a))^{1/3}\right)$.
\end{example}

\section{When is the $\gcd$ large?}
As we have seen,
when the sequence $(f^{\circ n}(a), g^{\circ n}(b))_n$ is not generic,
$\gcd(f^{\circ n}(a) - \alpha, g^{\circ n}(b) - \beta)$ might be big in general.
Our goal in this section is to show the following result.

\begin{theorem}
Assume Vojta's Conjecture. Suppose $f,g\in \mathbb{Z}[X]$ and $a,b,\alpha,\beta\in \mathbb{Z}$.
Then for all $\eta >0$, 
\begin{itemize}
\item either the set $$ \{n\in \mathbb{N}~|~\log\gcd(f^{\circ n}(a) - \alpha, g^{\circ n}(b) - \beta)\ge \eta \cdot d^n\}   $$
is a finite union of arithmetic progressions, or
\item  there is a finite union of arithmetic progressions $J$ such that
$$\lim_{n\rightarrow \infty, n\in J} \frac{1}{\eta d^n}\cdot {\log\gcd(f^{\circ n}(a) - \alpha, g^{\circ n}(b) - \beta)}  = 1.$$
\end{itemize}

\end{theorem}

\begin{proof}
   We choose $D$ as in the proof of Theorem $\ref{HGCD}$. That is,
we choose $D = D(\varepsilon, f, g,a,b)\in \mathbb{N}$ sufficiently large so that
$$
  \frac{M'}{d^D}\cdot \left({4} \hat{h}_f(a) + 4\hat{h}_g(b) + C \right)
 < \frac{\eta}{2} $$
where
\begin{equation*}
   \begin{aligned}
M' &= \max_{f^{\circ D} (x) = \alpha,~g^{\circ D}(y) = \beta}
\left( e_{x}(f^{\circ D} -\alpha), e_{y}(g^{\circ D} -\beta)\right).
   \end{aligned}
\end{equation*}
Then the proof of Theorem $\ref{HGCD}$ shows that assuming Vojta's Conjecture,
there is a proper algebraic subset $V\subseteq \airplaneOne$ such that
as long as $(f^{\circ(n-D)}(a), g^{\circ (n-D)}(b)) \notin V$ and $n$ is sufficiently large, we have
$$\log\gcd(f^{\circ n}(a) - \alpha, g^{\circ n}(b) - \beta) < \frac{\eta}{2}  \cdot d^n . $$
Let $I = \{n\in\mathbb{N}~|~(f^{\circ(n-D)}(a), g^{\circ (n-D)}(b)) \in V\}$.
Then the set
$$ \{n\in \mathbb{N}\setminus I~|~
\log\gcd(f^{\circ n}(a) - \alpha, g^{\circ n}(b) - \beta)\ge \eta \cdot d^n\}   $$
is finite.
By the Dynamical Mordell-Lang Theorem for polynomial maps on the affine plane (cf. \cite{Xie15}), $I$ is a finite union of arithmetic progressions.
Hence it suffices to show that the set
$$ \{n\in I~|
~\log\gcd(f^{\circ n}(a) - \alpha, g^{\circ n}(b) - \beta)\ge \eta \cdot d^n\}   $$
is a finite union of arithmetic progressions.
Looking at each irreducible component of $V$, it is enough to consider the case when $V$ is a curve.
In that case the set
$$\{(f^{\circ n}(a) - \alpha, g^{\circ n}(b) - \beta) ~|~n\in I\}$$
is contained in the curve $V':=f^{\circ (D)}(V)+(-\alpha,-\beta)$ where $+$ means translation on $\mathbb{A}^2$.
By abuse of notation, we also donote by $V'$ its Zariski closure in $\airplaneOne$.
Suppose $\iota: V'\hookrightarrow \airplaneOne$ is the inclusion map. 

Suppose $(x_1, x_2)\in V'$ and fix $D'\in \mathrm{Div}(V')$ of degree 1, then
\begin{equation*}
   \begin{aligned}
h_{\gcd}(x_1, x_2) &= h_{\airplaneOne, (0,0)} (x_1, x_2)       \\
&= h_{V',~ \iota^*(0,0)}(x_1, x_2) + O(1) \\
& = \deg(\iota^*(0,0))\cdot h_{V', D'} (x_1,x_2) + O(1)
   \end{aligned}
\end{equation*}
where the last equality follows from Proposition B.3.5 of \cite{HS00}, due originally to Siegel. 

Clearly it's enough to consider the case when $a$ is not preperiodic under $f$ and $b$ is not preperiodic under $g$.
In this case the projection $\pi_1: V'\rightarrow \mathbb{P}^1,~(x_1,x_2)\mapsto x_1$ is dominant.
Fix $D\in \mathrm{Div}(\mathbb{P}^1)$ of degree $1$.
Then $$h_{V', D'} (x_1,x_2)  =\frac{1}{\deg({\pi_1})}\cdot h_{\mathbb{P}^1, D}(x_1) + O(1) $$
by Theorem \ref{Weil}. 
Now
\begin{equation*}
   \begin{aligned}
      h_{\gcd}(f^{\circ n}(a) -\alpha, g^{\circ n}(b) - \beta)
      &= \frac{\deg(\iota^*(0,0))}{\deg(\pi_1)}\cdot h_{\mathbb{P}^1, D}(f^{\circ n}(a) -\alpha) + O(1) \\
      & = \frac{\deg(\iota^*(0,0))}{\deg(\pi_1)} \cdot \left( \hat{h}_f(a)\cdot d^n + O(1)\right) + O(1).
   \end{aligned}
\end{equation*}

Therefore, in the case when $V'$ is a curve, if $\displaystyle \eta = \hat{h}_f(a)  \cdot \frac{\deg(\iota^*(0,0))}{{\deg(\pi_1)}}$,
then $$\lim_{n\rightarrow \infty, n\in J}\frac{1}{\eta\cdot d^n} \cdot {\log\gcd(f^{\circ n}(a) - \alpha, g^{\circ n}(b) - \beta)}  = 1$$
for a finite union of arithmetic progression $J$;
otherwise the set
$$ \{n\in I~|
~\log\gcd(f^{\circ n}(a) - \alpha, g^{\circ n}(b) - \beta)\ge \eta \cdot d^n\}   $$
is always a finite set or complement of a finite set.
Hence for general $V'$, for all but finitely many $\eta$, the set in the statement is a finite union of arithmetic progressions.
\end{proof}

\section*{Acknowledgement}

I want to express my gratitude to my advisor Thomas Tucker for suggesting this project
and for many valuable discussions.
I also would like to thank Joseph Silverman for useful comments on an earlier draft of this paper. 
I would like to thank the referee for many useful comments and suggestions.  
In addition, I want to thank Shouman Das, Wayne Peng, and U\v{g}ur Yi\v{g}it for proofreading this paper.

\bibliography{refFile}{}
\bibliographystyle{amsalpha}

\end{document}